\documentclass[12pt,a4paper]{article}

\usepackage{latexsym}
\usepackage{amsmath}
\usepackage{amssymb}
\usepackage{arydshln}

\usepackage{cite}
\usepackage{stmaryrd}
\usepackage{enumerate}

\usepackage{hyperref}

\usepackage{color}
\usepackage{lineno}
\usepackage{graphicx}
\usepackage{ae}
\usepackage{amsmath}
\usepackage{amssymb}
\usepackage{latexsym}
\usepackage{url}
\usepackage{epsfig}
\usepackage{mathrsfs}
\usepackage{amsfonts}
\usepackage{amsthm}

\usepackage{float}
\usepackage{subfig}
\usepackage{tikz}
\usetikzlibrary{positioning}
\usepackage{multirow}

\newcommand{\Cay}{\mathrm{Cay}}
\newcommand{\GL}{\mathrm{GL}}
\newcommand{\Spec}{\mathrm{Spec}}

\newtheorem{theorem}{Theorem}[section]

\newtheorem{lemma}[theorem]{Lemma}

\newtheorem{cor}[theorem]{Corollary}

\theoremstyle{definition}

\numberwithin{equation}{section} 
\allowdisplaybreaks    

\def\qed{\hfill$\Box$\vspace{12pt}}

\long\def\delete#1{}

\usepackage{xcolor}
\usepackage[normalem]{ulem}

\begin{document}
\title {Integral Cayley graphs over a nonabelian group of order $8n$}

\author{Bei Ye$^{a,b}$,~Xiaogang Liu$^{a,b,c,}$\thanks{Supported by the National Natural Science Foundation of China (No. 12371358) and the Guangdong Basic and Applied
Basic Research Foundation (No. 2023A1515010986).}~$^,$\thanks{ Corresponding author. Email addresses: ybei@mail.nwpu.edu.cn, xiaogliu@nwpu.edu.cn}
\\[2mm]
{\small $^a$School of Mathematics and Statistics,}\\[-0.8ex]
{\small Northwestern Polytechnical University, Xi'an, Shaanxi 710072, P.R.~China}\\
{\small $^b$Research \& Development Institute of Northwestern Polytechnical University in Shenzhen,}\\[-0.8ex]
{\small Shenzhen, Guandong 518063, P.R. China}\\
{\small $^c$Xi'an-Budapest Joint Research Center for Combinatorics,}\\[-0.8ex]
{\small Northwestern Polytechnical University, Xi'an, Shaanxi 710129, P.R. China}\\
}
\date{}

\openup 0.5\jot
\maketitle

\begin{abstract}
A graph is called an integral graph when all eigenvalues of its adjacency matrix are integers. We study which Cayley graphs over a nonabelian group
$$
T_{8n}=\left\langle a,b\mid a^{2n}=b^8=e,a^n=b^4,b^{-1}ab=a^{-1} \right \rangle
$$
are integral graphs. Based on the group representation theory, we first give the irreducible matrix representations and characters of $T_{8n}$. Then we give necessary and sufficient conditions for which Cayley graphs over $T_{8n}$ are integral graphs. As applications, we also characterize some families of connected integral Cayley graphs over $T_{8n}$.

\smallskip


\end{abstract}

\section{Introduction}

Let $\Gamma$ be a simple undirected graph with $n$ vertices. The \emph{adjacency matrix} of $\Gamma$ is denoted by $A(\Gamma) =\left(a_{ij}\right)_{n \times n}$, where $a_{ij} = 1$ if vertices $i$ and $j$ are adjacent in $\Gamma$, and $a_{ij}=0$ otherwise. The \emph{adjacency spectrum} of $\Gamma$ is the multiset of the eigenvalues of $A(\Gamma)$.  If $\lambda_1,\lambda_2,\ldots,\lambda_r$ are the distinct eigenvalues of $A(\Gamma) $ and $m_1,m_2,\ldots,m_r$ are the corresponding multiplicities, then the adjacency spectrum of $\Gamma$ is denoted by
$$
\Spec (\Gamma) =\left\{[\lambda_1]^{m_1}, [\lambda_2]^{m_2}, \ldots, [\lambda_r]^{m_r}\right\},
$$
where we usually omit $m_i$ if $m_i = 1$ for some $i$.

The graph $\Gamma$ is called an \emph{integral graph}, if all eigenvalues of  $A(\Gamma)$ are integers. The concept of integral graph was originally proposed by Harary and Schwenk  in \cite{Harary}, where they proposed the problem: \textbf{Which  graphs are integral graphs}?	Since then, recognizing and constructing integral graphs have become an important research topic in spectral graph theory, and many families of integral graphs have been characterized \cite{Balinska, Bussemaker, Wang, Csikvari, Watanabe}.
	
Let $G$ be a finite group with the identity element $e$, and $\emptyset\not=S
\subseteq G \backslash \{e\}$ with $S=S^{-1}=\{s^{-1}\mid s\in S\}$ (called an \emph{inverse closed subset}). The Cayley graph $\Gamma = \Cay(G,S)$ is denoted to be the graph whose vertex set is $G$ and two vertices $x,y\in G$ are adjacent if and only if $xy^{-1}\in S$. In 1975, Lov\'asz \cite{Lovasz} discovered that the spectrum of Cayley graph can be calculated by using the group representation theory. In 1979,  Babai \cite{Babai} provided an expression for the spectrum of  $\Cay(G, S)$ by the irreducible characters of the finite group $G$. This is a remarkable achievement in computing the spectra of Cayley graphs for finite groups. However, it is still extremely difficult to determine the spectra of Cayley graphs due to unresolved issues in number theory. Until now, the eigenvalues of Cayley graphs have been investigated extensively, especially for integral Cayley graphs \cite{Abdollahi, CFH, Klotz, CFL1,CFL2,DeVos, Ilic, LHH, So, HL,WLW,  Estelyi}. For more results, please refer to an excellent survey \cite{Liu}.

In this paper, we characterize integral Cayley graphs over the nonabelian group $T_{8n}$, denoted by
 $$
 T_{8n} = \left\langle a,b \mid a^{2n}=b^8=e,a^n=b^4,b^{-1}ab=a^{-1} \right\rangle,
 $$
 which comes from \cite{CI-group}. We first give all the irreducible characters of the group $T_{8n}$ (See Tables \ref{Table1} and \ref{Table2}). Then we obtain a necessary and sufficient condition for $\Cay(T_{8n},S)$ to be integral by using the irreducible characters of $T_{8n}$ (See Theorem \ref{integrality}). We also give a necessary and sufficient condition for $\Cay(T_{8n},S)$ to be integral by using the atoms of the boolean algebra of  cyclic groups (See Theorem \ref{boolean-oddeven}). Finally, we characterize some integral normal Cayley graphs over $T_{8n}$ (See Theorem \ref{normal-integer}). Moreover, we also determine some families of connected integral Cayley graphs over $T_{8n}$ (See Corollaries \ref{cor2}, \ref{cor3} and \ref{cor4}).

\section{Preliminaries}




Let $G$ be a finite group and  $\mathrm{GL}(n,\mathbb{C})$ the \emph{general linear group} of $n \times n$ invertible matrices over the complex field $\mathbb{C}$. A \emph{matrix representation} $\rho$ of $G$ over $\mathbb{C}$ is a group homomorphism
	$$
	\rho : G \to \mathrm{GL}(n, \mathbb{C}),
	$$
where $n$ is called the \emph{degree} of the matrix representation $\rho$.

The \emph{character} $\chi_\rho$ of a matrix representation $\rho$ is the mapping $\chi_\rho : G \to \mathbb{C}$, defined by
	$$
	 \quad \chi_\rho(g) = \mathrm{Tr}(\rho(g))
	$$
	for all $g \in G$, where $\mathrm{Tr}$ denotes the \emph{trace} of the matrix. Here, $\chi_\rho(\mathrm{1}_{G})$ is called the \emph{degree} of the character $\chi_\rho$.

Let $\rho_{1}, \rho_{2}: G \to \mathrm{GL}(n,\mathbb{C})$ be two matrix representations. We say that $\rho_{1}$ and $\rho_{2}$ are \emph{equivalent}, written as $\rho_{1} \cong \rho_{2}$, if there exists a matrix $X \in \mathrm{GL}(n,\mathbb{C})$ such that
$$ X\rho_{1}(g)X^{-1} =\rho_{2}(g)$$
for all $g \in G$. The \emph{direct sum} $\rho_{1} \oplus \rho_{2}$ is defined to be the matrix representation of $G$ given by
$$ (\rho_{1} \oplus  \rho_{2})(g) = \rho_{1}(g)\oplus \rho_{2}(g) = \left(
\begin{array}{cc}
	\rho_{1}(g) & 0 \\
	0 & \rho_{2}(g)
\end{array}
\right)$$
for all $g\in G$.

A matrix representation $\rho$ is \emph{reducible} if $\rho \cong \rho_{1}\oplus \rho_{2}$, otherwise, $\rho$ is \emph{irreducible}. The character $\chi_{\rho}$ is \emph{irreducible} if the matrix representation $\rho$ is irreducible.

\begin{lemma}\emph{(See \cite[Theorem~3.1]{Babai})}\label{spetrum}
Let $G$ be a finite group of order $n$ with the identity element $e$, and $\emptyset\not = S
\subseteq G \backslash \{e\}$ with $S=S^{-1}$. Suppose that the irreducible characters (over $\mathbb{C}$) of $G$ are $\chi_{1},\chi_{2},\dots,\chi_{h}$ with respective degrees $n_{1},n_{2},\dots , n_{h}$. Then the spectrum of $\textup{Cay}(G,S)$ can be arranged as
$$ \textup{Spec}(\textup{Cay}(G,S)) =\{[\lambda_{11}]^{n_{1}}, \dots, [\lambda_{1n_{1}}]^{n_{1}}, \dots, [\lambda_{h1}]^{n_{h}}, \dots,[\lambda_{hn_{h}}]^{n_{h}} \}.$$
Furthermore,
$$ \lambda_{i1}^t+\lambda_{i2}^t+\cdots + \lambda_{in_{i}}^t = \sum_{s_{1},\ldots,s_{t}\in S}{\chi_{i}\left(\prod_{l=1}^{t}s_{l}\right)},
$$
for any positive integer $t$ and $i = 1,2,\dots , h$.
\end{lemma}

Let
$$
\{ a^{r},a^{r}b,a^{r}b^{2},a^{r}b^{3}\mid 0\leq r \leq 2n-1\}
$$
denote the set of all elements of the group $T_{8n}$. Then we obtain the following results by  straightforward computations.

\begin{lemma}\label{computations}
For the elements of $T_{8n}$, the following hold:
	\begin{itemize}
		\item[\rm (1)] $a^{r}b=ba^{-r}, ~ a^{r}b^2=b^2a^{r}$;
		\item[\rm (2)] $(a^rb)^{-1}=a^{r+n}b^3,~ (a^r b^2)^{-1}=a^{n-r}b^2,~ (a^{r}b^{3})^{-1}=a^{r+n}b$.
	\end{itemize}
	
\end{lemma}

\begin{lemma}\label{conjugacy}
	The $2n+6$ conjugacy classes of $T_{8n}$ are
	$$
	\{e\},~ \{a^{n}\},~  \{a^{r},a^{-r}\}~( 1\leq r \leq n-1),~ \{a^{2r}b\mid 0\leq r \leq n-1\},~ \{a^{2r+1}b\mid 0\leq r \leq n-1\},
	$$
	$$
	\{b^2\},~ \{a^{n}b^2\},~ \{a^{r}b^{2},a^{-r}b^{2}\}~( 1\leq r \leq n-1 ),~ \{a^{2r}b^{3}\mid 0\leq r \leq n-1\},~  \{a^{2r+1}b^{3}\mid 0\leq r \leq n-1\}.
	$$
\end{lemma}


\begin{lemma}\emph{(See \cite[Theorem~17.11]{representation})}\label{linear}
Let $G$ be a finite group. Then the number of distinct linear characters of $G$ is equal to $|G/G'|$, where
$$
G'=\left\langle [g,h]\mid g,h\in G, [g,h]=g^{-1}h^{-1}gh \right\rangle.
$$
\end{lemma}
\begin{lemma}\emph{(See \cite[Corollary~9.3]{representation})}\label{irreducible}
	Let $G$ be a finite group, and $\rho:G\rightarrow \GL(n,\mathbb{C})$ a matrix representation of  $G$. Then $\rho$ is irreducible if and only if every $n\times n$ matrix $A$ which satisfies $ (g\rho)A=A(g\rho)$ for all $g\in G$ has the form $A=\lambda I_{n}$ with $\lambda \in \mathbb{C}$, where $I_n$ is the identity matrix of order $n$.
\end{lemma}
\begin{lemma}\emph{(See \cite[Theorem~15.3]{representation})}\label{irreducible-number}
Let $G$ be a finite group. Then the number of irreducible characters of $G$ is equal to the number of conjugacy classes of $G$.
\end{lemma}

\begin{lemma}\label{characters-odd}
Let $n$ be  an odd integer. Then the irreducible matrix representations of $T_{8n}$ are $8$ linear representations $\rho_{1j}~ (j= 0, 1, \dots,7 )$, $n-1$ representations $ \rho_{2k}~ (k=1, 2, \dots , n-1 )$ of degree $2$, and  $n-1$ representations $\rho_{3h}~ (h= 1,2,\dots, n-1 )$ of degree $2$, which are shown as follows:
	\begin{itemize}
		\item[\rm (1)] $\rho_{1j}: a \mapsto \left(-1\right)^j, \quad b \mapsto \omega^{j}$;
		\item[\rm (2)] $\rho_{2k}: a \mapsto \left( \begin{array}{cc}
			\varepsilon ^k & 0\\
			0 & \varepsilon ^{-k}
		\end{array} \right), \quad b \mapsto
	\left(\begin{array}{cc}
		0& (-1)^k \\
		1 &0
	\end{array}
		\right)$;
		\item[\rm (3)] $\rho_{3h}: a \mapsto \left( \begin{array}{cc}
			\left(-1\right)^{h+1}\xi^{h} & 0\\
			0& \left(-1\right)^{h+1}\xi^{-h}
		\end{array} \right), \quad b \mapsto
		\left(\begin{array}{cc}
		0	& 1 \\
			\mathbf{i} &0
		\end{array}
		\right)$,
	\end{itemize}
where $\omega = \textup{exp}(\frac{\pi}{4}\mathbf{i})$, $\varepsilon =  \textup{exp}(\frac{2\pi}{n}\mathbf{i})$, $\xi= \textup{exp}(\frac{\pi}{n}\mathbf{i})$ and $\mathbf{i}^2=-1$.
\end{lemma}

\begin{proof}
It is easy to verify that $\rho_{1j}~ (j= 0, 1, \dots,7 )$, $ \rho_{2k}~ (k=1, 2, \dots , n-1 )$ and $\rho_{3h}~ (h= 1,2,\dots, n-1 )$ are all matrix representations of $G$. In the following, we prove that they are irreducible.

Note that
$$
T_{8n}'=\left\langle [g,h] \mid g,h\in T_{8n}, [g,h]=g^{-1}h^{-1}gh\right\rangle = \left \langle a^2\right\rangle
$$
and $ |T_{8n}/T_{8n}'|=8$. Then Lemma \ref{linear} implies that the number of distinct representations of degree $1$ is $8$. Clearly, such $8$ distinct representations are $\rho_{1j}~\left(j= 0,1,\dots,7 \right)$, as shown in $(1)$.

Let
$$
A = \left(\begin{array}{cc}
		\alpha & \beta \\
		\gamma & \delta
	\end{array}\right)
$$
be an invertible matrix over $\mathbb{C}$. Suppose that
$$ \left( \begin{array}{cc}
	\varepsilon ^k &0 \\
	0& \varepsilon ^{-k}
\end{array} \right) \left(\begin{array}{cc}
\alpha & \beta \\
\gamma & \delta
\end{array}\right) =  \left(\begin{array}{cc}
\alpha & \beta \\
\gamma & \delta
\end{array}\right)\left( \begin{array}{cc}
\varepsilon ^k & 0\\
0& \varepsilon ^{-k}
\end{array} \right) ,$$
and
$$ \left(\begin{array}{cc}
	0& (-1)^k \\
	1 &0
\end{array}
\right) \left(\begin{array}{cc}
	\alpha & \beta \\
	\gamma & \delta
\end{array}\right) =  \left(\begin{array}{cc}
	\alpha & \beta \\
	\gamma & \delta
\end{array}\right) \left(\begin{array}{cc}
	0& (-1)^k \\
	1 &0
\end{array}
\right).$$
Then
$$
\alpha = \delta,~ \beta = \gamma = 0 .
$$
Thus, every matrix $A$ which satisfies $(g\rho_{2k})A=A(g\rho_{2k})$ for all $g\in T_{8n}$ has the form $A=\alpha I_{2}$ with $\alpha \in \mathbb{C}$.
Therefore, by Lemma \ref{irreducible}, the representations $\rho_{2k}~(k = 1,2,\dots , n-1)$ for $\varepsilon = \text{exp}(\frac{2\pi}{n}\mathbf{i})$ are irreducible and inequivalent (their characters are distinct).


Similarly, we can verify that $\rho_{3h}~(h= 1,2,\dots, n-1 )$ are irreducible and inequivalent representations. Moreover, $\rho_{2k}$ and $\rho_{3h}$ are not equivalent, since $b^4 \in \mathrm{Ker} \rho_{2k}$ but $b^4 \notin \mathrm{Ker}\rho_{3h}$, where $\mathrm{Ker} \rho = \{g \in T_{8n} \mid \rho(g) = I\}$ and $I$ is the identity matrix.

	
Therefore, by Lemmas \ref{conjugacy} and \ref{irreducible-number}, we have  found all the irreducible matrix representations of $T_{8n}$.
\qed\end{proof}

\begin{lemma}\label{characters-even}
Let $n$ be an even integer. Then the irreducible matrix representations of $T_{8n}$ are shown as follows:
	\begin{itemize}
		\item[\rm (1)] $\rho_{1j}~(j=0,2,4,6): a \mapsto 1, \quad  b \mapsto \omega^{j}$;
		\item[\rm ] $\rho_{1j}~(j =1,3,5,7): a \mapsto -1, \quad b \mapsto \omega^{j-1}$;
		\item[\rm (2)] $\rho_{2k}~(k=2,4,\dots,n-2): a \mapsto \left( \begin{array}{cc}
			\xi ^k & 0\\
		0	& \xi ^{-k}
		\end{array} \right), \quad b \mapsto
		\left(\begin{array}{cc}
		0	& 1 \\
			1 &0
		\end{array}
		\right)$;
		\item[\rm ] $\rho_{2k}~(k=1,3,\dots,n-1): a \mapsto \left( \begin{array}{cc}
			\xi^{k} & 0\\
			0& \xi^{-k}
		\end{array} \right), \quad b \mapsto
		\left(\begin{array}{cc}
		0	& \mathbf{i} \\
			1 &0
		\end{array}
		\right)$;
		\item[\rm (3)]
		$\rho_{3h}~(h=2,4,\dots,n-2): a \mapsto \left( \begin{array}{cc}
			\xi ^{h} & 0\\
		0	& \xi ^{-h}
		\end{array} \right), \quad b \mapsto
		\left(\begin{array}{cc}
		0	& -1 \\
			1 &0
		\end{array}
		\right)$;
		\item[\rm ] $\rho_{3h}~(h=1,3,\dots,n-1): a \mapsto \left( \begin{array}{cc}
			\xi^{h} & 0\\
		0	& \xi^{-h}
		\end{array} \right), \quad b \mapsto
		\left(\begin{array}{cc}
		0	& -\mathbf{i} \\
			1 &0
		\end{array}
		\right)$,
	\end{itemize}
	where $\omega = \textup{exp}({\frac{\pi}{4}\mathbf{i}})$, $\xi= \textup{exp}({\frac{\pi}{n}\mathbf{i}})$ and $\mathbf{i}^2=-1$.
\end{lemma}

\begin{proof}
	The proof is similar to that of Lemma \ref{characters-odd}. Hence we omit the details here. \qed
\end{proof}

 By  Lemmas \ref{characters-odd} and \ref{characters-even}, we obtain the character tables of $T_{8n}$, which are given in Table \ref{Table1} for odd $n$ and Table \ref{Table2} for  even $n$.

\begin{table}[H]
	\caption{\textbf{ Character Table of $T_{8n}$ for Odd $n$ }}
	\label{Table1}
	\centering
	\begin{tabular}{c|ccccc}
		\hline
		 & $1$ & $a^r$ & $a^{r}b$ & $a^{r}b^2$ & $a^{r}b^3$  \\
		\hline
		$\chi_{j}$ & \multirow{2}{*}{$1$} & \multirow{2}{*}{$(-1)^{rj}$} & \multirow{2}{*}{$(-1)^{rj}\omega^j$} &  \multirow{2}{*}{$(-1)^{rj}\omega^{2j}$} & \multirow{2}{*}{$(-1)^{rj}\omega^{3j}$} \\
		$(j=0,1,\dots,7)$ & \\
		
		$ \varphi_{k}$ & \multirow{2}{*}{$2$} & \multirow{2}{*}{$\varepsilon^{kr} + \varepsilon^{-kr}$} & \multirow{2}{*}{$0$} & \multirow{2}{*}{$(-1)^{k}(\varepsilon^{kr} + \varepsilon^{-kr})$} & \multirow{2}{*}{$0$} \\
		$(k=1,2,\dots,n-1)$ & \\
		
		$ \psi_{h}$ & \multirow{2}{*}{$2$} & \multirow{2}{*}{$(-1)^{rh+r}(\xi^{hr} + \xi^{-hr})$} & \multirow{2}{*}{$0$} & \multirow{2}{*}{$\mathbf{i}(-1)^{rh+r}(\xi^{hr} + \xi^{-hr})$} & \multirow{2}{*}{$0$} \\
		$(h=1,2,\dots,n-1)$ & \\
		\hline
	\end{tabular}
	
\end{table}
\begin{table}[H]
	\caption{\textbf{ Character Table of $T_{8n}$ for Even $n$ }}
	\label{Table2}
	\centering
	\begin{tabular}{cc|ccccc}
		\hline
 	  &	& $1$ & $a^r$ & $a^{r}b$ & $a^{r}b^2$ & $a^{r}b^3$  \\
		\hline
		\multirow{2}{*}{$\chi_{j}$} & $(j=0,2,4,6)$ & \multirow{2}{*}{1} & 1 & $\omega^j$ &  $\omega^{2j}$ & $\omega^{3j}$ \\
		& $(j =1,3,5,7)$ &  & $(-1)^r$ & $(-1)^r\omega^{j-1}$ &  $(-1)^r\omega^{2(j-1)}$ & $(-1)^r\omega^{3(j-1)}$ \\
		\hline
		\multirow{2}{*}{$\varphi_{k}$} & $(k=2,4,\dots,n-2)$ &\multirow{2}{*}{2} & \multirow{2}{*}{$\xi^{kr}+\xi^{-kr}$}  & \multirow{2}{*}{0} & $\xi^{kr}+\xi^{-kr} $ & \multirow{2}{*}{0} \\
		& $(k=1,3,\dots,n-1)$ &  &  &  & $\mathbf{i}(\xi^{kr}+\xi^{-kr})$ &  \\
		\hline
		\multirow{2}{*}{$ \psi_{h}$} & $(h=2,4,\dots,n-2)$ & \multirow{2}{*}{2} & \multirow{2}{*}{$\xi^{hr} + \xi^{-hr}$} & \multirow{2}{*}{0} & $-(\xi^{hr} + \xi^{-hr})$ & \multirow{2}{*}{0} \\
		& $(h=1,3,\dots,n-1)$ &  &   &   & $-\mathbf{i}(\xi^{hr} + \xi^{-hr})$ &   \\
		\hline
	\end{tabular}
	\end{table}

\begin{lemma}\emph{(See\cite[Lemma~2.1]{Abdollahi})}\label{2n-root}
	Let $\xi=\textup{exp}(\frac{\pi}{n}\mathbf{i})$, where $ \mathbf{i}^2=-1$. Then
	\begin{itemize}
	\item [\rm (1)] $\sum\limits_{j=1}^{2n-1}\xi^j=-1,~\sum\limits_{j=1}^{2n-1}(\xi^j+\xi^{-j})=-2$;
	\item [\rm (2)] If $l$ is even, for $0 \leq l \leq 2n-1$, then $\sum\limits_{j=1}^{n-1}\xi^{lj}=-1 , ~\sum\limits_{j=1}^{n-1}(\xi^{lj}+\xi^{-lj})=-2$.
	\end{itemize}
\end{lemma}

By straightforward computations, one can easily verify the following results.

\begin{lemma}\label{n-root}
	Let $\varepsilon=\textup{exp}(\frac{2\pi}{n}\mathbf{i})$, where $ \mathbf{i}^2=-1$. Then
	\begin{itemize}
	\item [\rm (1)] $\sum\limits_{j=1}^{n-1}\varepsilon^j=\sum\limits_{j=1}^{2n-1}\varepsilon^j=-1 , ~\sum\limits_{j=1}^{n-1}(\varepsilon^j+\varepsilon^{-j})=\sum\limits_{j=1}^{2n-1}(\varepsilon^j+\varepsilon^{-j})=-2$;
	\item [\rm (2)] If $l$ is even, for $0 \leq l \leq 2n-1$, then $\sum\limits_{j=1}^{n-1}\varepsilon^{lj}=-1, ~\sum\limits_{j=1}^{n-1}(\varepsilon^{lj}+\varepsilon^{-lj})=-2$.
	\end{itemize}
\end{lemma}


Let $G$ be a finite group and $\mathcal{F}$ a family of subgroups of $G$. The \emph{Boolean algebra} $B(G)$ generated by $\mathcal{F}$ is the set which elements are obtained by arbitrary finite intersections, unions, and complements of the sets in the family $\mathcal{F}$.


\begin{lemma}\emph{(See\cite{boolean-algebra})}\label{BG}
	Let $G$ be an abelian group with the identity element $e$, and $\emptyset\not=S
\subseteq G \backslash \{e\}$ with $S=S^{-1}$. Then the Cayley graph $\textup{Cay}(G,S)$ is integral if and only if $S\in B(G)$.
\end{lemma}



\begin{lemma}\emph{(See\cite{representation})}\label{circulant}
Let $\mathbb{Z}_{n} = \langle a \rangle$ be a cyclic group of order $n$. Then the irreducible characters of $\mathbb{Z}_{n}$ are $\rho_{k}: a^g \mapsto \varepsilon^{kg} $, where $\varepsilon=\textup{exp}(\frac{2\pi}{n}\mathbf{i})$ and $k,g = 0, 1,\dots, n-1$. Moreover, the Cayley graph $\textup{Cay}(\mathbb{Z}_{n},S)$ (that is circulant graph) has eigenvalues:
$$
\lambda_{\rho_{k}}=\sum\limits_{s\in S}\rho_{k}(s)=\sum\limits_{s\in S}\varepsilon^{ks} \text{~~for~} k= 0,1,\dots,n-1.
$$
\end{lemma}

	Let $G$ be a finite group and $S$ a nonempty subset of $G$. The \emph{Cayley digraph} $\Cay(G, S)$ over $G$ with connection set $S$ is defined to be the digraph whose vertex set is $G$ and $\left(x,y\right)$ is a directed arc if and only if $xy^{-1}\in S$. The Cayley digraph $\Cay(G, S)$ is called \emph{normal} if  $S$ is a union of conjugacy classes of $G$. We say that $S$ is \emph{power-closed} if, for every $x \in S$ and $y \in \langle x  \rangle $ with $ \langle y  \rangle = \langle x  \rangle $, it follows that $ y \in S$.
\begin{lemma}\emph{(See\cite{normal})}\label{digraphs}
	Let $G$ be a finite group and $S$ a union of conjugacy classes of $G$. Then the normal Cayley digraph $\textup{Cay}(G,S)$ is integral if and only if $S$ is power-closed.
\end{lemma}


\section{A necessary and sufficient condition for integral Cayley graphs over $T_{8n}$}

In this section, we give a necessary and sufficient condition for Cayley graphs over $T_{8n}$ to be integral graphs.

Let $G$ be a group and $\chi$ a character of $G$. Let $A$ and $B$ be two subsets of $G$. For convenience, we define
$$
\chi(\emptyset)=0, \text{ } \chi(A) = \sum_{a \in A}\chi(a), \text{  and  } \chi(AB) = \sum_{a\in A, b \in B}\chi(ab).
$$


%

\begin{theorem}\label{integrality}
Let $S=S_{1}\cup S_{2} \subseteq T_{8n} \backslash \{e\}$ such that $  S = S^{-1}$, where $S_{1} \subseteq \langle a \rangle \cup \langle a  \rangle b^2$ and $ S_{2} \subseteq \langle a  \rangle b \cup  \langle a  \rangle b^3$. Then $\Cay(T_{8n},S)$ is integral if and only if the following conditions hold for $j=0,1,\dots,7$ and $ k,h=1,2,\dots,n-1$:
	\begin{itemize}
	\item [\rm (1)] $\chi_{j}(S)$ is an integer;
	\item [\rm (2)] $\varphi_{k}(S_1)$ and $\varphi_{k}(S_{1}^2) + \varphi_{k}(S_{2}^2)$ are integers;
	\item [\rm (3)] $\triangle _{\varphi_{k}}(S) = 2(\varphi_{k}(S_{1}^2)+\varphi_{k}(S_{2}^2)) - \varphi_{k}^2(S_{1})$ is a perfect square;
	\item [\rm (4)] $\psi_{h}(S_1)$ and $\psi_{h}(S_{1}^2) + \psi_{h}(S_{2}^2)$ are integers;
	\item [\rm (5)] $\triangle _{\psi_{h}}(S) = 2(\psi_{h}(S_{1}^2)+\psi_{h}(S_{2}^2)) - \psi_{h}^2(S_{1})$ is a perfect square.
	\end{itemize}
\end{theorem}

\begin{proof}
By Lemma \ref{computations}, we have
$$
S_{1}S_{2}=\left\{s_{1}s_{2}\mid s_{1}\in S_{1}, s_{2}\in S_{2}\right\} \subseteq \langle a \rangle b \cup \langle a \rangle b^3,
$$
$$
S_{2}S_{1}=\left\{s_{2}s_{1}\mid s_{2}\in S_{2}, s_{1}\in S_{1}\right\} \subseteq \langle a \rangle b \cup \langle a  \rangle b^3,
$$
$$
S_{1}^2=\left\{s_{1}s_{2}\mid s_{1},s_{2}\in S_{1}\right\} \subseteq  \langle a \rangle \cup \langle a \rangle b^2,
$$
$$
S_{2}^2=\left\{s_{1}s_{2}\mid s_{1},s_{2}\in S_{2}\right\} \subseteq  \langle a \rangle \cup  \langle a  \rangle b^2.
$$
Then by Tables \ref{Table1} and \ref{Table2}, we have
	$$
	\sum\limits_{s_{2}\in S_{2}}\varphi_{k}(s_{2}) = \sum\limits_{s\in S_{1},t\in S_{2}}\varphi_{k}(st) =\sum\limits_{s\in S_{2},t\in S_{1}}\varphi_{k}(st) = 0,
	$$
	$$
	\sum\limits_{s_{2}\in S_{2}}\psi_{h}(s_{2}) = \sum\limits_{s\in S_{1},t\in S_{2}}\psi_{h}(st)= \sum\limits_{s\in S_{2},t\in S_{1}}\psi_{h}(st) = 0,
	$$
i.e.,
	$$ \varphi_{k}(S_{2}) = \varphi_{k}(S_{1}S_{2}) = \varphi_{k}(S_{2}S_{1}) = \psi_{h}(S_{2}) = \psi_{h}(S_{1}S_{2}) = \psi_{h}(S_{2}S_{1}) = 0.$$
 	Thus,
	$$
	\varphi_{k}(S) =  \sum\limits_{s_{1}\in S_{1}}\varphi_{k}(s_{1}) +\sum\limits_{s_{2}\in S_{2}}\varphi_{k}(s_{2}) = \sum\limits_{s_{1}\in S_{1}}\varphi_{k}(s_{1}) = \varphi_{k}(S_{1}),
	$$
	\begin{align*}
	\varphi_{k}(S^2) &= \sum\limits_{s,t\in S}\varphi_{k}(st)= \varphi_{k}(S_{1}^2) + \varphi_{k}(S_{1}S_{2}) + \varphi_{k}(S_{2}S_{1}) +\varphi_{k}(S_{2}^2) \\
	&= \varphi_{k}(S_{1}^2) + \varphi_{k}(S_{2}^2).
    \end{align*}
    Similarly, we have
	$$
	\psi_{h}(S) =  \psi_{h}(S_{1}),
	$$
	$$
	\psi_{h}(S^2)= \psi_{h}(S_{1}^2) + \psi_{h}(S_{2}^2).
	$$
By Lemma \ref{spetrum},  the eigenvalues of $\Cay(T_{8n}, S)$ satisfy the following equations:
	\begin{equation}\label{jkh}
		\begin{cases}
			\lambda_{j} = \chi_{j}(S),  &  j = 0,1,\dots, 7;  \\
			\mu_{k1} + \mu_{k2} = \varphi_{k}(S_{1}), &  k = 1, 2, \dots , n-1 ;\\
			\mu_{k1}^2 + \mu_{k2}^2=\varphi_{k}(S_{1}^2) + \varphi_{k}(S_{2}^2), &  k = 1, 2, \dots , n-1 ;\\
			\eta_{h1} + \eta_{h2}=\psi_{h}(S_{1}),  & h = 1, 2, \dots , n-1 ;\\
			\eta_{h1}^2 + \eta_{h2}^2=\psi_{h}(S_{1}^2) + \psi_{h}(S_{2}^2),  &  h = 1, 2, \dots , n-1 .\\
		\end{cases}
	\end{equation}
Suppose that $\Cay(T_{8n},S)$ is integral, that is all eigenvalues of  $\Cay(T_{8n},S)$ are integers. By \eqref{jkh}, $\chi_{j}(S)$, $\varphi_{k}(S_{1})$, $ \varphi_{k}(S_{1}^2) + \varphi_{k}(S_{2}^2)$, $\psi_{h}(S_{1})$ and $\psi_{h}(S_{1}^2) + \psi_{h}(S_{2}^2)$ must be integers, so (1), (2) and (4) hold.

Note that $\mu_{k1}$ and $\mu_{k2}$ are two integral roots of the equation
$$
x^2 - (\mu_{k1} + \mu_{k2})x + \mu_{k1}\mu_{k2}  = 0,
$$
which is equivalent to
\begin{equation}\label{quadratic}
	x^2 -  \varphi_{k}(S_{1}) x + \frac{1}{2}\left(\varphi_{k}^2(S_{1})-\left(\varphi_{k}(S_{1}^2) + \varphi_{k}(S_{2}^2)\right)\right) = 0.
\end{equation}
 Thus, the discriminant $\triangle _{\varphi_{k}}(S) = 2(\varphi_{k}(S_{1}^2)+\varphi_{k}(S_{2}^2)) - \varphi_{k}^2(S_{1})$ must be a perfect square, and then (3) holds.

Similarly, one can easily verify (5) by using that $\eta_{k1}$ and $\eta_{k2}$ are two integral roots of the equation
$$
x^2 - (\eta_{k1} +\eta_{k2})x + \eta_{k1}\eta_{k2}  = 0.
$$

On the contrary, suppose that (1)--(5) hold. Clearly, by \eqref{jkh}, $\lambda_{j}$ is an integer. Note that the solutions $\mu_{k1}$ and $\mu_{k2}$ of \eqref{quadratic} are rational.  Thus, $\mu_{k1} $ and $\mu_{k2}$ must be  integers. Similarly, $\eta_{h1}$ and $\eta_{h2}$ are integers. Hence, $\Cay(T_{8n}, S)$ is integral.
\qed\end{proof}

\begin{cor}\label{cor2}
	Let $S=S_{1}\cup S_{2}\subseteq T_{8n} \backslash \{e\}$ such that $S = S^{-1}$, where $S_{1} = \langle a^2 \rangle\backslash\{e\} \cup \langle a  \rangle b^2 $ and $ S_{2} = \langle a  \rangle b \cup \langle a \rangle b^3$. Then $\Cay(T_{8n}, S)$ is a connected integral graph, whose spectrum is
$$\Spec (\Cay(T_{8n}, S)) = \left\{[7n-1], [-n-1]^3, [n-1]^4, [-1]^{8(n-1)}\right\}.$$
\end{cor}
\begin{proof}
Note that $\left\langle S \right\rangle = T_{8n}$. Then $\Cay(T_{8n},S)$ is connected. In the following, we verify that $\Cay(T_{8n},S)$ is integral by Tables \ref{Table1}, \ref{Table2} and Theorem \ref{integrality}. Notice that
	\begin{align}\label{d1j}
		\chi_{j}(S)  = \chi_{j}(S_{1}) + \chi_{j}(S_{2}) = \sum\limits_{r=1}^{n-1}\chi_{j}(a^{2r}) + \sum\limits_{l=0}^{2n-1}\chi_{j}(a^{l}b^2) + \sum\limits_{l=0}^{2n-1}\chi_{j}(a^{l}b) +\sum\limits_{l=0}^{2n-1}\chi_{j}(a^{l}b^3) ,
	\end{align}
	\begin{align}\label{d2k1}
 \varphi_{k}(S_{1})
		= \sum\limits_{r=1}^{n-1}\varphi_{k}(a^{2r}) + \sum\limits_{l=0}^{2n-1}\varphi_{k}(a^{l}b^2),
	\end{align}
\begin{align}\label{S1square}
\varphi_{k}(S_{1}^2)	=\sum\limits_{r_{1},r_{2}=1}^{n-1}\varphi_{k}(a^{2r_{1}}a^{2r_{2}}) + \sum\limits_{l_{1},l_{2}=0}^{2n-1}\varphi_{k}(a^{l_{1}}b^2a^{l_{2}}b^2) + \sum\limits_{r=1}^{n-1}\sum\limits_{l=0}^{2n-1}\varphi_{k}(a^{2r}a^{l}b^2) +\sum\limits_{l=0}^{2n-1}\sum\limits_{r=1}^{n-1}\varphi_{k}(a^{l}b^2a^{2r}),
\end{align}
\begin{align}\label{S2square}
\varphi_{k}(S_{2}^2)=	\sum\limits_{l_{1},l_{2}=0}^{2n-1}\varphi_{k}(a^{l_{1}}ba^{l_{2}}b) + \sum\limits_{l_{1},l_{2}=0}^{2n-1}\varphi_{k}(a^{l_{1}}b^3a^{l_{2}}b^3) + \sum\limits_{l_{1},l_{2}=0}^{2n-1}\varphi_{k}(a^{l_{1}}ba^{l_{2}}b^3) +\sum\limits_{l_{1},l_{2}=0}^{n-1}\varphi_{k}(a^{l_{1}}b^3a^{l_{2}}b).
\end{align}
Then, we  consider the following two cases.

\noindent \emph{Case 1.} $n$ is odd. In this case,  by Lemma \ref{n-root}, we have
		\begin{align*}
		\chi_{j}(S)
			& = \sum\limits_{r=1}^{n-1}(-1)^{2rj} + \sum\limits_{l=0}^{2n-1}(-1)^{lj}\omega^{2j} + \sum\limits_{l=0}^{2n-1}(-1)^{lj}\omega^{j} + \sum\limits_{l=0}^{2n-1}(-1)^{lj}\omega^{3j} 	\\
			& = \left\{\begin{array}{ll}
              n-1, &  j = 1,3,5,7, \\[0.1cm]
				-n-1, & j = 2,4,6, \\[0.1cm]
				7n-1, & j = 0,
			   \end{array}\right.\\[0.1cm]
 \varphi_{k}(S_{1}) & = \sum\limits_{r=1}^{n-1}(\varepsilon^{2rk}+\varepsilon^{-2rk}) + \sum\limits_{l=0}^{2n-1}(-1)^k(\varepsilon^{lk} + \varepsilon^{lk})  \\
		& = 2 \mathrm{Re} \left(\sum\limits_{r=1}^{n-1}\varepsilon^{2rk}\right) + 2(-1)^k\mathrm{Re}\left(\sum\limits_{l=0}^{2n-1}\varepsilon^{lk}\right) \\
		& = -2 + 0 \\
		& =-2,
	\end{align*}
where $\mathrm{Re}(x)$ represents the real part of the complex number $x$,
\begin{align*}
	\varphi_{k}(S_{1}^2) & =\sum\limits_{r_{1},r_{2}=1}^{n-1}(\varepsilon^{2(r_{1}+r_{2})k}+\varepsilon^{-2(r_{1}+r_{2})k})+\sum\limits_{l_{1},l_{2}=0}^{2n-1}(\varepsilon^{(l_{1}+l_{2}+n)k}+\varepsilon^{-(l_{1}+l_{2}+n)k}) \\ & \quad +2\sum\limits_{r=1}^{n-1}\sum\limits_{l=0}^{2n-1}(-1)^{k}(\varepsilon^{(2r+l)k}+\varepsilon^{-(2r+l)k})  \\
  & = 2\mathrm{Re}(\sum\limits_{r=1}^{n-1}\varepsilon^{2rk})^2+2\mathrm{Re}(\sum\limits_{l=0}^{2n-1}\varepsilon^{lk})^2+2(-1)^k\mathrm{Re}(\sum\limits_{r=1}^{n-1}\varepsilon^{2rk}\sum\limits_{l=0}^{2n-1}\varepsilon^{lk}) \\
  & = 2,
\end{align*}
\begin{align}
	\varphi_{k}(S_{2}^2) & =\sum\limits_{l_{1},l_{2}=0}^{2n-1}(-1)^k(\varepsilon^{(l_{1}-l_{2})k}+\varepsilon^{-(l_{1}-l_{2})k}) + \sum\limits_{l_{1},l_{2}=0}^{2n-1}(-1)^k(\varepsilon^{(l_{1}-l_{2}+n)k}+\varepsilon^{-(l_{1}-l_{2}+n)k}) \nonumber \\ & \quad +2\sum\limits_{l_{1},l_{2}=0}^{2n-1}(\varepsilon^{(l_{1}-l_{2}+n)k}+\varepsilon^{-(l_{1}-l_{2}+n)k}) \nonumber \\
	& = 4(-1)^k\mathrm{Re}(\sum\limits_{l_{1},l_{2}=0}^{2n-1}\varepsilon^{(l_{1}-l_{2})k}) + 4\mathrm{Re}(\sum\limits_{l_{1},l_{2}=0}^{2n-1}\varepsilon^{(l_{1}-l_{2})k})\nonumber \\
\label{S2square-odd} 	& = 0.
\end{align}
Then $\varphi_{k}(S_{1}^2)+\varphi_{k}(S_{2}^2) = 2$ and
 $\triangle _{\varphi_{k}}(S) = 2(\varphi_{k}(S_{1}^2)+\varphi_{k}(S_{2}^2)) - \varphi_{k}^2(S_{1}) = 0 $, which is a perfect square.

Similarly, one can easily verify that
$$ \psi_{h}(S_1) = -2,\quad  \psi_{h}(S_{1}^2) + \psi_{h}(S_{2}^2) = 2,\quad  \triangle _{\psi_{h}}(S) = 0.$$
Therefore, by Theorem \ref{integrality}, $\Cay(T_{8n}, S)$ is an integral graph if $n$ is odd.

\noindent \emph{Case 2.} $n$ is even. In this case,  by Lemma \ref{2n-root}, we have
\begin{align*}
\chi_{j}(S)  & = \begin{cases}
		\sum\limits_{r=1}^{n-1}1 + \sum\limits_{l=0}^{2n-1}\omega^{2j} + \sum\limits_{l=0}^{2n-1}\omega^{j} + \sum\limits_{l=0}^{2n-1}\omega^{3j}, & j = 0,2,4,6, 	\\
		\sum\limits_{r=1}^{n-1}(-1)^{2r}+ \sum\limits_{l=0}^{2n-1}(-1)^l\omega^{2(j-1)} + \sum\limits_{l=0}^{2n-1}(-1)^l\omega^{j-1} + \sum\limits_{l=0}^{2n-1}(-1)^l\omega^{3(j-1)},  & j = 1,3,5,7,
	\end{cases}  \\
& = \begin{cases}
	n-1, &  j = 1,3,5,7 ,\\
	-n-1, & j = 2,4,6 ,\\
	7n-1, & j = 0,
\end{cases}\\
	 \varphi_{k}(S_{1})  & = \begin{cases}
		\sum\limits_{r=1}^{n-1}(\xi^{2rk}+\xi^{-2rk}) + \sum\limits_{l=0}^{2n-1}(\xi^{lk} + \xi^{lk}), &  k = 2,4,\dots , n-2, \\
		\sum\limits_{r=1}^{n-1}(\xi^{2rk}+\xi^{-2rk}) + \sum\limits_{l=0}^{2n-1}\mathbf{i}(\xi^{lk} + \xi^{lk}), &  k = 1,3,\dots, n-1,
	\end{cases} \\
& = \begin{cases}
 	-2,	&  k = 2,4,\dots, n-2, \\
 	-2, &  k = 1,3,\dots, n-1,
\end{cases}
\end{align*}
\begin{align*}
	& \varphi_{k}(S_{1}^2) \\
	& = \begin{cases}
		2\mathrm{Re}\sum\limits_{r_{1},r_{2}=1}^{n-1}(\xi^{2(r_{1}+r_{2})k}) + 2\mathrm{Re}\sum\limits_{l_{1},l_{2}=0}^{2n-1}(\xi^{(l_{1}+l_{2}+n)k}) + 4\mathrm{Re}\sum\limits_{r=1}^{n-1}\sum\limits_{l=0}^{2n-1}(\xi^{(2r+l)k}), &  k = 2,4,\dots,n-2, \\
		2\mathrm{Re}\sum\limits_{r_{1},r_{2}=1}^{n-1}(\xi^{2(r_{1}+r_{2})k}) + 2\mathrm{Re}\sum\limits_{l_{1},l_{2}=0}^{2n-1}(\xi^{(l_{1}+l_{2}+n)k}) + 4\mathrm{Re}\sum\limits_{r=1}^{n-1}\sum\limits_{l=0}^{2n-1}\mathbf{i}(\xi^{(2r+l)k}), & k = 1,3,\dots,n-1,
	\end{cases}   \\
&= \begin{cases}
	2, & k = 2,4,\dots,n-2 ,\\
	2, & k = 1,3,\dots,n-1 ,
\end{cases}
\end{align*}
\begin{align}
	\varphi_{k}(S_{2}^2) &= \begin{cases}
		2\mathrm{Re}\sum\limits_{l_{1},l_{2}=0}^{2n-1}(\xi^{(l_{1}-l_{2})k}) + 6\mathrm{Re}\sum\limits_{l_{1},l_{2}=0}^{2n-1}(\xi^{(l_{1}-l_{2}+n)k}), &  k = 2,4,\dots,n-2, \\
		2\mathrm{Re}\sum\limits_{l_{1},l_{2}=0}^{2n-1}\mathbf{i}(\xi^{(l_{1}-l_{2})k}) + (2\mathbf{i}+4)\mathrm{Re}\sum\limits_{l_{1},l_{2}=0}^{2n-1}(\xi^{(l_{1}-l_{2}+n)k}), & k = 1,3,\dots,n-1,
	\end{cases} \nonumber \\
\label{S2square-even} &= \begin{cases}
	0, & k = 2,4,\dots,n-2, \\
	0, & k = 1,3,\dots,n-1.
\end{cases}
\end{align}
Then
\begin{align*}
\varphi_{k}(S_{1}^2) + \varphi_{k}(S_{2}^2)  = \begin{cases}
	2, & k = 2,4,\dots,n-2, \\
	2, & k = 1,3,\dots,n-1,
\end{cases}
\end{align*}
and
 $\triangle _{\varphi_{k}}(S) = 2(\varphi_{k}(S_{1}^2)+\varphi_{k}(S_{2}^2)) - \varphi_{k}^2(S_{1}) = 0 $, which is a perfect square.

Similarly, one can easily verify that
\begin{align*}
	\psi_{h}(S_1) & = \begin{cases}
		-2,	&  h = 2,4,\dots , n-2, \\
		-2, &  h = 1,3,\dots , n-1,
	\end{cases}\\
 	\psi_{h}(S_{1}^2) + \psi_{h}(S_{2}^2)
	& = \begin{cases}
		2, & h = 2,4,\dots,n-2 ,\\
		2, & h = 1,3,\dots,n-1,
	\end{cases}
\end{align*}
 and $ \triangle _{\psi_{h}}(S) =2(\psi_{h}(S_{1}^2)+\psi_{h}(S_{2}^2)) - \psi_{h}^2(S_{1})= 0$, which is a perfect square. Therefore, by Theorem \ref{integrality}, $\Cay(T_{8n}, S)$ is an integral graph if $n$ is even.

Finally, by \eqref{jkh}, we can determine the spectrum of $\Cay(T_{8n}, S)$ immediately.
\qed\end{proof}

\begin{cor}
	For each positive integer $n$, there is a connected $(7n-1)$-regular integral graph with $8n$ vertices.
	
\end{cor}

\section{Integral Cayley graphs over $T_{8n}$ by Boolean algebra of cyclic groups}

In this section, we study the integrality of Cayley graphs over $T_{8n}$ by using the Boolean algebra of cyclic groups.

\begin{lemma}\label{a2}
	Let $T \subseteq \langle a^2 \rangle \subseteq T_{8n} $ such that $T\in B(\langle a^2 \rangle )$ and $ T=T^{-1}$. Then $2\varphi_{k}(T^2) = \varphi_{k}^2(T)$ and $2\psi_{h}(T^2) = \psi_{h}^2(T)$  for $k,h = 1,2,\dots, n-1$.
\end{lemma}

\begin{proof}
Suppose that $T = \left\{a^{2d}\mid d \in \Omega\right\}$,  where $\Omega = -\Omega$ is a set of integers. Then $T^2 = \left\{a^{2(s+t)} \mid s,t \in \Omega\right\}$. Consider the following two cases.
	
\noindent \emph{Case 1.} $n$ is odd. In this case,
		\begin{align*}
			2\varphi_{k}(T^2) &= 2 \sum\limits_{s,t\in \Omega}\varphi_{k}\left(a^{2(s+t)}\right) \\
			& = 2 \sum\limits_{s,t\in\Omega}\left(\varepsilon^{2(s+t)k}+\varepsilon^{-2(s+t)k}\right) \\
			& = \sum\limits_{s\in\Omega}\left(\varepsilon^{2sk} + \varepsilon^{-2sk}\right) \sum\limits_{t\in \Omega}\left(\varepsilon^{2tk} + \varepsilon^{-2tk}\right) \\
			& = \varphi_{k}^2(T),\\
		2\psi_{h}(T^2) &= 2 \sum\limits_{s,t\in \Omega}\psi_{h}\left(a^{2(s+t)}\right) \\
		& = 2 \sum\limits_{s,t\in\Omega}(-1)^{2(s+t)(h+1)}\left(\xi^{2(s+t)h}+\xi^{-2(s+t)h}\right) \\
	& =  \sum\limits_{s\in\Omega}\left(\xi^{2sk} + \xi^{-2sk}\right) \sum\limits_{t\in \Omega}\left(\xi^{2tk} + \varepsilon^{-2tk}\right) \\
		& = \psi_{h}^2(T).
	\end{align*}
\noindent \emph{Case 2.} $n$ is even. In this case,
\begin{align*}
		2\varphi_{k}(T^2) &= 2 \sum\limits_{s,t\in \Omega}\varphi_{k}\left(a^{2(s+t)}\right)  = 2 \sum\limits_{s,t\in\Omega}\left(\xi^{2(s+t)k}+\xi^{-2(s+t)k}\right)  = \varphi_{k}^2(T).
\end{align*}
Note that $\varphi_{k}(T)=\psi_{h}(T)$ if $k=h$. Then $ 2\psi_{h}(T^2) = \psi_{h}^2(T)$.

Therefore, $2\varphi_{k}(T^2) = \varphi_{k}^2(T)$ and $2\psi_{h}(T^2) = \psi_{h}^2(T)$ for $k,h = 1,2,\dots, n-1$.
\qed\end{proof}

	\begin{theorem}\label{boolean-oddeven}
		Let $S = S_{1}\cup S_{2}\subseteq T_{8n}\backslash \{e\}$ with $S = S^{-1}$, where $ S_{1}\subseteq \langle a^2 \rangle $ and $S_{2} \subseteq  \langle a  \rangle b \cup \langle a \rangle b^3$. Then $\Cay(T_{8n},S)$ is integral if and only if the following conditions hold for $ j = 0,1,\dots ,7$ and $ k,h = 1,2,\dots n-1$:
	\begin{itemize}
		\item[\rm (1)] $S_{1} \in B(\left\langle a^2 \right\rangle )$;
		\item[\rm (2)]  $\chi_{j}(S)$ and  $\varphi_{k}(S_{1})$ are integers;
		\item[\rm (3)] $2\varphi_{k}(S_{2}^2)$ and $2\psi_{h}(S_{2}^2)$ are perfect squares.
	\end{itemize}
	\end{theorem}
	\begin{proof}
By Lemma \ref{circulant}, let $\rho_{l}:a^{2r} \mapsto  \xi^{2rl}$ be the irreducible characters of cyclic group $\langle a^2 \rangle $, where $\xi=\textup{exp}(\frac{\pi}{n}\mathbf{i})$ and $l,r = 0,1, \dots , n-1$. Assume that $S_{1} = \left\{a^{2d}\mid d \in \Omega\right\}$, where $\Omega = -\Omega$ is a set of integers. Then $S_{1}^2 = \left\{a^{2(s+t)} \mid  s,t \in \Omega\right\}$ is inverse-closed. Thus, there exists a multiset $U=-U$ of integers such that $S_{1}^2 =\left\{a^{2u}\mid u \in U\right\}$. By Tables \ref{Table1} and \ref{Table2}, we have
		\begin{align}
\nonumber\psi_{h}\left(S_{1}\right) &= \sum\limits_{s\in \Omega}\psi_{h}\left(a^{2s}\right)  \\
\nonumber& = \begin{cases}
			 	\sum\limits_{s\in \Omega}\left(-1\right)^{2s(h+1)}\left(\xi^{2sh}+\xi^{-2sh}\right), & \text{ if $n$ is odd},  \\ 			 	\sum\limits_{s\in \Omega}\left(\xi^{2sh}+\xi^{-2sh}\right),  & \text{ if $n$ is even},
			 \end{cases} \\
\nonumber& =  2\sum\limits_{s \in \Omega}(\xi^{2sh}) \\
\label{S1Use1-2-1}& = 2\rho_{h}\left(S_{1}\right).
		\end{align}
Similarly, we have
\begin{equation}\label{S1Algeb1-1-1}
	 \psi_{h}\left(S_{1}^2\right) = 2\rho_{h}\left(S_{1}^2\right).
 \end{equation}

Suppose that $\Cay(T_{8n},S)$ is integral. By Theorem \ref{integrality}, $\chi_{j}(S)$, $\varphi_{k}(S_{1})$ and $\psi_{h}(S_{1})$ are integers, and then (2) holds.

Recall that $\psi_{h}(S_{1})~(h= 1,2, \ldots, n-1)$ are integers. By \eqref{S1Use1-2-1}, $\rho_{h}(S_{1})~(h= 1, 2, \ldots, n-1)$ are rational numbers. By  Lemma \ref{circulant}, $\rho_{h}(S_{1})~(h= 1, \ldots, n-1)$ are eigenvalues of $\Cay(\left\langle a^2 \right\rangle, S_{1})$, and then $\rho_{h}(S_{1})~(h= 1, \ldots, n-1)$ must be integers. Thus, $\Cay(\left\langle a^2 \right\rangle, S_{1})$ is an integral graph. By Lemma \ref{BG}, $S_{1} \in B(\left\langle a^2\right\rangle)$, yielding (1).

By Lemma \ref{a2}, we have $2\varphi_{k}(S_{1}^2)=\varphi_{k}^2(S_{1})$ and $2\psi_{h}(S_{1}^2)=\psi_{h}^2(S_{1})$. Then
\begin{eqnarray}
\label{3.11-use-1} \triangle_{\varphi_{k}}(S) = 2(\varphi_{k}(S_{1}^2)+\varphi_{k}(S_{2}^2))-\varphi_{k}^2(S_{1}) = 2\varphi_{k}(S_{2}^2),\\
\label{3.11-use-2} \triangle_{\psi_{h}}(S) = 2(\psi_{h}(S_{1}^2)+\psi_{h}(S_{2}^2))-\psi_{h}^2(S_{1}) = 2\psi_{h}(S_{2}^2).
\end{eqnarray}
By Theorem \ref{integrality}, $2\varphi_{k}(S_{2}^2)$ and $ 2\psi_{h}(S_{2}^2) $ must be perfect squares. So (3) holds.

Conversely, suppose that (1), (2) and (3) hold. By (1) and Lemma \ref{BG}, $\Cay(\left\langle a^2 \right\rangle,S_{1})$ is integral. By Lemma \ref{circulant} and \eqref{S1Use1-2-1},  $\rho_{h}(S_{1})$ and $\psi_{h}(S_{1}) = 2\rho_{h}(S_{1})$ are integers. Thus, Lemma \ref{a2} implies that $\psi_{h}(S_{1}^2)=\frac{1}{2}\psi_{h}^2(S_{1})$ is a rational number. Note \eqref{S1Algeb1-1-1} that $\psi_{h}(S_{1}^2)$ is an algebraic number. Then $\psi_{h}(S_{1}^2)$ must be an integer.

Similarly, note (2) that $\varphi_{k}(S_{1})$ is an integer. By Lemma \ref{a2}, $\varphi_{k}(S_{1}^2) = \frac{1}{2}\varphi_{k}^2(S_{1})$ is a rational number, which is also an algebraic integer. Thus,  $\varphi_{k}(S_{1}^2) $ is an integer.

By (3), \eqref{3.11-use-1} and \eqref{3.11-use-2}, $\triangle _{\varphi_{k}}(S)$ and $\triangle _{\psi_{h}}(S)$ are perfect squares, and then $\varphi_{k}(S_{2}^2)$ and $\psi_{h}(S_{2}^2)$ are integers. So $ \varphi_{k}(S_{1}^2) + \varphi_{k}(S_{2}^2) $ and $\psi_{h}(S_{1}^2) + \psi_{h}(S_{2}^2) $ are integers.

Therefore, by Theorem \ref{integrality}, $\Cay(T_{8n},S)$ is integral.
	\qed\end{proof}

	\begin{cor}\label{cor3}
		Let $S=S_{1}\cup S_{2}\subseteq T_{8n} \backslash \{e\}$ such that $S = S^{-1} $, where $S_{1}=\langle a^2 \rangle \backslash \{e\} $ and $S_{2} = \langle a \rangle b \cup \langle a  \rangle b^3$. Then $\Cay(T_{8n},S)$ is a connected integral graph, whose  spectrum is
		$$\Spec (\Cay(T_{8n}, S)) = \left\{[-3n-1],[5n-1],[n-1]^6,[-1]^{8(n-1)}\right\}.$$
	\end{cor}

\begin{proof}
Note that $\langle S \rangle = T_{8n}$. Then $\Cay(T_{8n},S)$ is connected. In the following, we verify that $\Cay(T_{8n}, S)$ is integral by Tables \ref{Table1}, \ref{Table2} and Theorem \ref{boolean-oddeven}. Notice that $S_{1} \in B( \langle a^2 \rangle ) $. Then
		\begin{align*}
			\chi_{j}(S) &= \chi_{j}(S_{1}) + \chi_{j}(S_{2}) \\
 &= \sum\limits_{r=1}^{n-1}\chi_{j}(a^{2r}) +  \sum\limits_{l=0}^{2n-1}\chi_{j}(a^{l}b) +\sum\limits_{l=0}^{2n-1}\chi_{j}(a^{l}b^3) \\
			& =\begin{cases}
				\sum\limits_{r=1}^{n-1}(-1)^{2rj}  + \sum\limits_{l=0}^{2n-1}(-1)^{lj}\omega^{j} + \sum\limits_{l=0}^{2n-1}(-1)^{lj}\omega^{3j}, & \text{if $n$ is odd and $j=0,1,\ldots,7$}, \\
				\sum\limits_{r=1}^{n-1}1  + \sum\limits_{l=0}^{2n-1}\omega^{j} + \sum\limits_{l=0}^{2n-1}\omega^{3j}, & \text{if $n$ is even and $ j = 0,2,4,6$}, \\
				\sum\limits_{r=1}^{n-1}(-1)^{2r}  + \sum\limits_{l=0}^{2n-1}(-1)^{l}\omega^{j-1} + \sum\limits_{l=0}^{2n-1}(-1)^{l}\omega^{3(j-1)} ,& \text{if $n$ is even and $ j = 1,3,5,7$},
			\end{cases} 	\\
			& = \begin{cases}
				n-1, &  j = 1,2,3,5,6,7, \\
				-3n-1, & j = 4, \\
				5n-1, & j = 0,
			\end{cases}
		\end{align*}
and
		\begin{align*}
			\varphi_{k}(S_{1}) &= \sum\limits_{r=1}^{n-1}(a^{2r}) \\
 &= \begin{cases}
				\sum\limits_{r=1}^{n-1}(\varepsilon^{2kr}+\varepsilon^{-2kr}),  & \text{if $n$ is odd}, \\
				\sum\limits_{r=1}^{n-1}(\xi^{2kr}+\xi^{-2kr}), & \text{if $n$ is even},
			\end{cases}\\
	&= \begin{cases}
		-2, &  \text{if $n$ is odd}, \\
		-2, & \text{if $n$ is even}.
	\end{cases}
		\end{align*}
By \eqref{S2square-odd} and \eqref{S2square-even}, we have	$\varphi_{k}(S_{2}^2) = 0$. Similar to \eqref{S2square-odd} and \eqref{S2square-even}, one can verify that $\psi_{h}(S_{2}^2) = 0 $. Then $2\varphi_{k}(S_{2}^2)$ and $2\psi_{h}(S_{2}^2)$ are perfect squares. Therefore, by Theorem \ref{boolean-oddeven}, $\Cay(T_{8n},S)$ is integral.

Finally, by Lemma \ref{a2}, we have $\varphi_{k}(S_{1}^2) = \frac{1}{2}\varphi_{k}^2(S_{1})=2$. Then $\varphi_{k}(S_{1}^2)+\varphi_{k}(S_{2}^2) = 2$. Similarly, one can verify that $\psi_{h}(S_{1})=-2$, $\psi_{h}(S_{1}^2)=\frac{1}{2}\psi_{h}^2(S_{1})=2$ and $\psi_{h}(S_{1}^2)+\psi_{h}(S_{2}^2)  = 2$.
Therefore, by \eqref{jkh}, we can determine the spectrum of $\Cay(T_{8n}, S)$ immediately.
	\qed\end{proof}

\begin{cor}
For each positive integer $n$, there is a connected $(5n-1)$-regular integral graph with $8n$ vertices.
	\end{cor}


\section{Integral normal Cayley graphs over $T_{8n}$}

In this section, we characterize integral normal Cayley graphs over $T_{8n}$. Before proceeding, define
$$
[x] = \left\{y \mid y\in \langle x  \rangle ,  \langle y \rangle =  \langle x  \rangle \right\}
$$
and let $\overline{a}$ denote the conjugacy class of $a\in T_{8n}$. Table \ref{Order} gives the orders of elements in $T_{8n}$, where $\mathrm{gcd}(x,y)$ denotes the greatest common divisor of $x$ and $y$, $\mathrm{lcm}(x,y)$ denotes the least common multiple of $x$ and $y$.

\begin{table}[ht]
	\caption{\textbf{ Orders of Elements in $T_{8n}$ }}
	\label{Order}
	\centering
	\begin{tabular}{c|c|c}
		\hline
		Type & Elements & Orders   \\
		\hline
		I & $a^i$ & $\frac{2n}{\mathrm{gcd}\left(i,2n\right)}$  \\
		\hline
		II & $ a^ib^j $~($j = 2$) & $ \mathrm{lcm}\left(\frac{2n}{\mathrm{gcd}(i,2n)}, 4\right) $\\
		\hline
		III & $ a^ib^j $~($j = 1,3$)  &  8  \\
		\hline
	\end{tabular}
\end{table}

	\begin{theorem}\label{normal-integer}
		Let $S= S_{1}\cup S_{2} \subseteq T_{8n} \backslash \{e\}$ such that $S=S^{-1}$, where $S_{1} \subseteq \langle a \rangle \cup
		\langle a \rangle b^2$ and $S_{2} \subseteq \langle a \rangle b \cup \langle a \rangle b^3$. Then the normal Cayley graph $\Cay(T_{8n},S)$ is integral if and only if $S_1$ and $S_2$ satisfy the following conditions:
		\begin{itemize}
			\item [\rm (1)] $S_{1}$ is one of the following sets: $ \emptyset $, $[a^r] $, $[a^rb^2]$, or arbitrary finite unions of these sets, where
			 $[a^r] = \{a^l \in \langle a^r \rangle \mid \gcd(k,o(a^r))=1 \textup{~and~} l = kr \pmod {2n},  k \in \mathbb{N}^{+} \}$,
			$[a^rb^2] = \{a^lb^2 \in  \langle a^rb^2  \rangle \mid \textup{lcm}(o(a^r),4)=\textup{lcm}(o(a^l),4),~ l = r(4t+1) \pmod{2n} \textup{~or~} l = r(4t+3)+n \pmod{2n}, t \in \mathbb{N}^{+}\}$, $o(a^r)=\frac{2n}{\textup{gcd}(r,2n)}$ denotes the order of the element $a^r$ and $o(a^rb^2) = \textup{lcm}(o(r),4)$;
			\item [\rm (2)] $S_{2}$  is one of the following sets:
\begin{itemize}
  \item[$\bullet$] $\emptyset $ or $ \{a^rb,a^rb^3\mid r=0,1,\dots , 2n-1 \}$, if $n$ is odd;
  \item[$\bullet$] $\emptyset$, $ \{a^rb,a^rb^3\mid r=0,2,\dots , 2n-2 \}$, $ \{a^rb,a^rb^3\mid r=1,3,\dots , 2n-1 \}$
			or $ \{a^rb,a^rb^3\mid r=0,1,\dots, 2n-1 \}$, if $n$ is even.
\end{itemize}
		\end{itemize}
	\end{theorem}

		
\begin{proof}		
By Lemma \ref{digraphs}, $\Cay(T_{8n},S)$ is an integral normal Cayley graph if and only if $S$ is power-closed. In the following, we give all possible cases for $S$ to be power-closed.

\noindent\emph{Case 1.} If $a^r,a^rb^2 \notin S_{1}$, then $S_{1} = \emptyset$.

\noindent\emph{Case 2.} If $a^r \in S_{1}$, then all generators of the cyclic subgroup $\langle a^r \rangle $ must belong to $S_{1}$. Suppose that $\langle a^l \rangle = \langle a^r \rangle $. Since
$$\langle a^r \rangle = \left\{a^r,a^{2r},\dots,a^{o(a^r)}\right\},$$
we have $l = kr \pmod {2n} $ for $k \in \mathbb{N}^{+}$, and $o(a^l) = o(a^r)$. Then $\gcd(k,o(a^r))=1$ must hold, since $o(a^r) = \frac{2n}{\gcd(r,2n)}$.
So we obtain $[a^r]$ as show in (1).

\noindent\emph{Case 3.} If $a^rb^2 \in S_{1}$, then all generators of the cyclic subgroup $\langle a^rb^2 \rangle $ must belong to $S_{1}$ and $\langle a^l \rangle \neq \langle a^rb^2 \rangle $ for all $a^l \in \langle a \rangle $. By Lemma \ref{computations}, we have $a^rb^2 = b^2a^r$ and $(a^rb^2)^k = a^{rk}b^{2k}$. Assume that $\langle a^lb^2 \rangle = \langle a^rb^2 \rangle $. Then, there exists $k \in \mathbb{N}^{+}$ such that $(a^rb^2)^k = a^{rk}b^{2k} = a^lb^2$. This implies that exactly one of the following must hold:
$$
\begin{cases}
	b^{2k} = b^2, \\
	l \equiv rk \pmod {2n},
\end{cases}\text{or~~}
\begin{cases}
	b^{2k} = a^nb^2, \\
	l \equiv rk+n \pmod {2n}.
\end{cases}
$$
Consequently, we have $2k = 8t+2$ or $2k=8t+4+2$ for $t \in \mathbb{N}^{+}$. Then,
$$ l\equiv r(4t+1) \pmod {2n} ~~\text{or}~~ l \equiv r(4t+3) + n \pmod {2n}$$
holds for $t \in \mathbb{N}^{+}$.
Finally, to make $o(a^lb^2) = o(a^rb^2)$, $\text{lcm}(o(a^r),4)=\text{lcm}(o(a^l),4)$ is necessary, since
$$o(a^rb^2) = \text{lcm}(\frac{2n}{\gcd(r,2n)},4) = \text{lcm}(o(r),4).$$
So we obtain $[a^rb^2]$ as show in (1).

\noindent\emph{Case 4.} If $a^rb,a^rb^3 \notin S_{2}$, then $S_{2} = \emptyset$.

\noindent\emph{Case 5.} $a^rb \in S_{2}$ or $a^rb^3 \in S_{2}$.
Note that
$$ \langle a^rb \rangle = \left\{a^rb, b^2, a^rb^3, a^n, a^{n+r}b, a^nb^2, a^{n+r}b^3, e\right\} = \langle a^rb^3 \rangle = \langle a^{n+r}b \rangle =  \langle a^{n+r}b^3 \rangle , $$
$$ \langle b^2 \rangle = \left\{b^2,a^n,a^nb^2,e\right\} =\langle a^nb^2 \rangle \neq \langle a^rb \rangle ,$$
$$ \langle a^n \rangle = \left\{a^n,e\right\}  \neq \langle a^rb \rangle.$$
Then $a^rb$, $a^rb^3,a^{n+r}b,a^{n+r}b^3\in S_{2}$, since $S_{2}$ is power-closed.

If $n$ is odd, by Lemma \ref{conjugacy},
$$\left\{a^rb, a^rb^3\mid r = 0,1,\dots,2n-1\right\} = \langle a\rangle b\cup \langle a \rangle b^3$$
is the unique union of conjugacy classes containing $a^rb$, $a^rb^3$, $a^{r+n}b$ and $a^{r+n}b^3$, which is power-closed.

If $n$ is even, by Lemma \ref{conjugacy}, one can easily verify that
$$ \overline{a^rb} \cup \overline{a^rb^3} \cup \overline{a^{r+n}b} \cup \overline{a^{r+n}b^3} = \left\{a^rb,a^rb^3\mid r=0,2,\dots , 2n-2\right\} $$
and
$$ \overline{a^rb} \cup \overline{a^rb^3} \cup \overline{a^{r+n}b} \cup \overline{a^{r+n}b^3} = \left\{a^rb,a^rb^3\mid r=1,3,\dots , 2n-1\right\}$$
are two unions of conjugacy classes containing $a^rb$, $a^rb^3$, $a^{r+n}b$ and $a^{r+n}b^3$, which are power-closed. Clearly, the union of such two sets is also power-closed.

This completes the proof.
	\qed\end{proof}

	\begin{cor}\label{cor4}
		Let $S= S_{1}\cup S_{2}\subseteq T_{8n} \backslash \{e\}$ such that $S=S^{-1}$, where $S_{1} = \{b^2,a^nb^2\}$ and $S_{2} = \langle a \rangle b \cup \langle a \rangle b^3$. Then the normal Cayley graph $\Cay(T_{8n},S)$ is a connected  integral graph, whose spectrum is
$$
\Spec(\Cay(T_{8n},S)) = \left\{[-4n+2],[4n+2],[2]^{2n-2},[-2]^{2n},[0]^{4n}\right\}.
$$
	\end{cor}

\begin{proof}
Note that $\langle S \rangle = T_{8n}$. Then $\Cay(T_{8n},S)$ is connected. Notice that $S_{1} = \left\{b^2, a^nb^2\right\}$, $\langle b^2 \rangle = \left\{b^2 , a^n , a^nb^2, e\right\}$, where $n = 0(4t+3)+n$ and $\text{lcm}(o(a^0),4) = \text{lcm}(o(a^n),4)$, so $[b^2] = \{b^2,a^nb^2\}$. Notice also that $S_{2} = \langle a \rangle b \cup \langle a \rangle b^2 = \{a^rb,a^rb^3 \mid r = 0,1,\ldots,2n-1\}$. Thus, by Theorem \ref{normal-integer}, $\Cay(T_{8n},S)$ is integral.
		
		Finally, by Tables \ref{Table1}, \ref{Table2}, \eqref{S2square-odd} and \eqref{S2square-even}, we have
		\begin{align*}
			\lambda_{j} &= \begin{cases} \omega^{2j}+\left(-1\right)^j\omega^{2j} + \sum\limits_{r=0}^{2n-1}\left(-1\right)^{rj}\omega^j + \sum\limits_{r=0}^{2n-1}\left(-1\right)^{rj}\omega^{3j}, & \text{if $n$ is odd and $j=0,1,\ldots,7$,} \\
			2\omega^{2j}  +	\sum\limits_{r=0}^{2n-1}\omega^j + 	\sum\limits_{r=0}^{2n-1}\omega^{3j},  &  \text{if $n$ is even  and $j = 0,2,4,6$,} \\
			2\omega^{2(j-1)}  +	\sum\limits_{r=0}^{2n-1}(-1)^r\omega^{(j-1)} + 	\sum\limits_{r=0}^{2n-1}(-1)^r\omega^{3(j-1)},  & \text{if $n$ is even and $j = 1,3,5,7$,}  \\
				\end{cases} \\
			&= \begin{cases}
				\begin{cases}
					0 ,& j = 1,3,5,7, \\
					-2, & j =2,6,  \\
					4n+2 ,& j =0, \\
					-4n+2 ,& j = 4,
				\end{cases}  &   \text{if  $n$ is odd}, \\
			\begin{cases}
				-2 ,& j =2,3,6,7, \\
				2,& j = 1,5 ,\\
				4n+2 ,& j =0, \\
				-4n+2 ,& j = 4,
			\end{cases}   &   \text{if $n$ is even},
			\end{cases}
		\end{align*}
	\begin{align*}
	\varphi_{k}(S_{1}) &=\varphi_{k}(b^2) + \varphi_{k}(a^nb^2) \\
		&= \begin{cases}
			(-1)^k\left(\varepsilon^{0k}+\varepsilon^{-0k}\right) + (-1)^k\left(\varepsilon^{nk}+\varepsilon^{-nk}\right),   & \text{if $n$ is odd and $k=1,2,\ldots,n-1$},  \\
			(\xi^{0k}+\xi^{-0k}) + (\xi^{nk}+\xi^{-nk}), & \text{if $n$ is even and $k = 2,4,\dots, n-2$},  \\
			\mathbf{i}(\xi^{0k}+\xi^{-0k}) + \mathbf{i}(\xi^{nk}+\xi^{-nk}), & \text{if $n$ is even and $ k = 1,3,\dots, n-1$},
		\end{cases}  \\
&=	\begin{cases}
		 4(-1)^k,  & \text{if $n$ is odd and $k=1,2,\ldots,n-1$},  \\
		4, & \text{if $n$ is even and $k = 2,4,\dots, n-2$}, \\
		0, & \text{if $n$ is even and $k = 1,3,\dots, n-1$},
	\end{cases}
	\end{align*}
\begin{align*}
\varphi_{k}(S_{1}^2) + \varphi_{k}(S_{2}^2)  & = 2\varphi_{k}(a^n) + 2\varphi_{k}(e) + \varphi_{k}(S_{2}^2) \\
	& = \begin{cases}
		2\left(\varepsilon^{nk}+\varepsilon^{-nk}\right) + 2\times2 + 0, & \text{if $n$ is odd}, \\
			2\left(\xi^{nk}+\xi^{-nk}\right) + 2\times2 + 0, & \text{if $n$ is even}, \\
	\end{cases} \\
	& = \begin{cases}
	8,	& \text{if $n$ is odd and $k=1,2,\ldots,n-1$}, \\
	8, & \text{if $n$ is even and $k = 2,4,\dots, n-2$}, \\
	0, & \text{if $n$ is even and $k = 1,3,\dots, n-1$}.
	\end{cases}
\end{align*}
Similarly, we have
	\begin{align*}
		\psi_{h}(S_{1}) &=  \psi_{h}(b^2) + \psi_{h}(a^nb^2) \\
& =
		\begin{cases}
			0, & \text{if $n$ is odd and $h=1,2,\ldots,n-1$}, \\
			-4 , & \text{if $n$ is even and $h = 2,4,\dots, n-2$}, \\
			0, &\text{if $n$ is even and $h = 1,3,\dots, n-1$},
		\end{cases}
	\end{align*}
	\begin{align*}
		\psi_{h}(S_{1}^2) + \psi_{h}(S_{2}^2)  & = 2\psi_{h}(a^n) + 2\psi_{h}(e) + \psi_{h}(S_{2}^2) \\
		& = \begin{cases}
			0,	& \text{if $n$ is odd and $h=1,2,\ldots,n-1$}, \\
			8,  & \text{if $n$ is even and $h = 2,4,\dots, n-2$}, \\
			0, &\text{if $n$ is even and $h = 1,3,\dots, n-1$}.
		\end{cases}
	\end{align*}
Therefore, by \eqref{jkh}, we can determine the spectrum of $\Cay(T_{8n}, S)$ immediately.
	\qed\end{proof}
 \begin{cor}
 	For each positive integer $n$, there is a connected $(4n+2)$-regular integral graph with $8n$ vertices.
 	
 \end{cor}
\section{Conclusion}
In this paper, we characterize integral Cayley graphs over $T_{8n}$. We first give all the irreducible characters of $T_{8n}$. Then we obtain a necessary and sufficient condition for $\Cay(T_{8n},S)$ to be integral by using the irreducible characters of $T_{8n}$. We also give a necessary and sufficient condition for $\Cay(T_{8n},S)$ to be integral by using the boolean algebra of cyclic groups. Finally, we characterize some integral normal Cayley graphs over $T_{8n}$. Moreover, we also determine some families of connected integral Cayley graphs over $T_{8n}$. As mentioned in the introduction, recognizing and constructing integral graphs is an important research topic in spectral graph theory. So we propose the following problem: \textbf{Characterize integral Cayley graphs over other kinds of non-abelian groups}.

\end{document}